\documentclass[a4paper, 12pt]{article}

\usepackage[utf8]{inputenc}
\usepackage[T1]{fontenc}
\usepackage{amsmath, amssymb}
\usepackage{amsthm}
\usepackage{graphicx}
\usepackage{tikz}
\usetikzlibrary{positioning}
\usepackage{hyperref}
\usepackage{geometry}
\title{A mathematical modelling of rheumatoid arthritis.}
\author{ P. Azerad, O. Iosifescu, A. Parmeggiani, O. Bortolotti,  G. Courties, F. Apparailly. }
\date{\today}
\usepackage{authblk}

\author[1]{P.~Azerad\thanks{Corresponding author.}}
\author[1]{O.~Iosifescu}
\author[2]{A.~Parmeggiani}
\author[3]{O.~Bortolotti}
\author[3,4]{F.~Apparailly}
\author[3]{G.~Courties}

\affil[1]{IMAG, UMR5149, Université de Montpellier, France}
\affil[2]{L2C, UMR5221, Université de Montpellier, France}
\affil[3]{IRMB, UMR1183, Université de Montpellier, INSERM, France}
\affil[4]{Clinical Department for osteoarticular diseases, Montpellier University Hospital Lapeyronie, France}

\newtheorem{rem}{Remark}[section]

\newtheorem{proposition}{Proposition}[section]
\newcommand{\keywords}[1]{%
  \par\noindent\textbf{Keywords: }#1
}
\newcommand{\msc}[1]{%
  \par\noindent\textbf{MSC Classification: }#1
}

\begin{document}

\maketitle 
\begin{abstract}
We developed a mathematical model to describe the inflammatory dynamics 
during arthritis development and resolution, focusing on the interaction between three key immune cell populations 
within the joint tissue, namely neutrophils (N),  sublining synovial tissue-resident macrophages (SL-TRM, T),
 and circulating monocytes that are recruited into the inflamed joint and differentiate into monocyte-derived macrophages (MoMac, M).
An inflammatory signal (S), representing the response to an external arthritogenic stimulus, 
drives the recruitment and activation of neutrophils and MoMac, while SL-TRM exert a regulatory and protective role.
We propose a minimal ordinary differential equation model of acute inflammatory response into a joint that couples an inflammatory signal S to the three cell populations (N, T, M) and incorporates recruitment, stimulation, nonlinear self-limitation and mutual regulation.
 We show that the equilibrium structure depends on three adimensional parameters A, B, C and derive a cubic equation characterizing nontrivial steady states.
We perform an analytical study of equilibria and their stability, and identify a bifurcation at A = 1 separating a disease-free regime from a sustained inflammatory state. 
Numerical simulations illustrate typical transient dynamics under various scenarios and show the ability of the model to capture the transition from acute to chronic inflammation, as well as the resolution of inflammation under certain conditions. 
In conclusion, our model provides a framework to explore mechanisms underlying acute versus sustained joint inflammation and offers a basis for parameter estimation from experimental data, as well as future extensions toward evaluation of therapeutic strategies.

\end{abstract}
\msc{92C50, 92D25, 34D20, 34C23, 37N25.}
\section{Introduction.}
Mathematical modelling of  auto-immune disease is relatively recent.
According to the systematic survey on mathematical modelling of auto-immune disease by  Ugolkov \textit{et al.}~\cite{frontiers2024},
the ﬁrst papers date back to the 1970s. Since
then, the number of relevant articles increased exponentially,
reaching more than 200 publications per year by 2024. 
However according to Macfarlane \textit{et al.}~\cite{macfarlane2020} which provides another careful and exhaustive bibliographic survey in 2020, the number of  mathematical models of Rheumatoid Arthritis (RA) is still relatively small.
An illuminative survey paper introducing the issues and pitfalls of mathematical modelling in immunology was from Andrew \textit{et al.}~\cite{andrew}. 
Let us cite Handel \textit{et al.}~\cite{Handel2020} for  examples and guidelines for sound mathematical modelling in immunology and also for the free software R-package DSAIRM~\cite{dsairm2019} which allows to easily simulate and analyze a large variety of mathematical models of immune response.
Among mathematical models that deal with aspects
of RA, we can mention the following ones.
In Seymour \textit{et al.}~\cite{seymour2001} a four differential equations model was introduced to describe the dynamics of the population of monocytes and the concentrations of three cytokines, namely IL-1, TNF-$\alpha$ and IL-10. 
In Jit \textit{et al.}~\cite{jit2005} RA versus systemic inflammatory response syndrom was modelled by a system of four differential equations, describing the evolution of population of certain cell compartments, namely TNF-$\alpha$ (free or bound), antibodies and antibody-antigen complexes.
The effect of TNF-$\alpha$ inhibitors like (s)TNFR2, Etanecerp, Infliximab was investigated.
Baker \textit{et al.}~\cite{baker2013} introduced a model with pro-inﬂammation 
(p) and anti-inﬂammation cells (a) and considered a system of two coupled differential equations for p and a. 
More complex models exhibiting various types of  bifurcations were introduced in~\cite{baker2017}. 
In Odisharia \textit{et al.}~\cite{jaiani2019} a five differential equations system was investigated to simulate the effect of tocilizumab on the populations of B and T-lymphocytes and cell population of cartilage.  
There exist also spatialized models, see Friedman \textit{et al.}~\cite{friedman} where partial differential equations were used to model a joint (synovial fluid and cartilage) and  Macfarlane \textit{et al.}~\cite{macfarlane2022}
 where an hybrid model was introduced to describe pannus formation in the arthritic joint.
Effects of drugs such as inﬂiximab, which blocks TNF-$\alpha$; methotrexate
which blocks cell growth, and tocilizumab
which blocks the receptor of IL-6 were investigated.
In Relouw \textit{et al.}~\cite{relouw} a fifteen equations model with 48 parameters was introduced to model sepsis.
In Zhou \textit{et al.}~\cite{zhou2018} a four differential equations model was used to describe the dynamics of the population of macrophages and fibroblasts  and their mutual growth factors.
A paper that particularly inspired us 
is from Ciuperca \textit{et al.}~\cite{torres2024} which focused on Alzheimer disease and presented a minimal compartment model with  five equations describing the concentrations of 
 A$\beta$-oligomers,
oligomers in the amyloid plaques,
A$\beta$-monomers,
microglial cells, and
interleukins.

Here we chose a compromise between simplicity and accuracy   to catch essential features of cell populations dynamics in the response to an acute inflammation signal. We chose to focus on cell populations and their interactions instead of the fine description of the many cytokines dynamics involved.
 Specifically, our model, presented in the next section, has four equations and only seven parameters and is an ordinary differential equations systems. We show that the behavior of our system depends on a single nondimensional parameter.
Despite its simplicity (or due to), we suggest that such a model provides a better understanding of the arthritis process and opens up interesting avenues for the design of future therapies.\par
The present paper is organized as follows.
In Section~\ref{sec:model}, we present the model. In Section~\ref{sec:analysis}, we analyse the mathematical properties of the differential equations system, compute its steady states and discuss their stability. In Section~\ref{sec:numerical_simulations}, we present some numerical simulations pertaining to different scenarii.
Section~\ref{sec:insights_predictions} discusses biological insights and testable predictions derived from the model.
The paper ends with  Section~\ref{sec:conclusion} that includes concluding remarks and  perspectives.

\section{Biological Model.}\label{sec:model}
In the experimental mouse model of arthritis used in this study, namely the serum transfer arthritis (STA) model, see~\cite{STA} and~\cite{weihaupt26},
 the acute inflammatory response is triggered by an external signal (the injection of an arthritogenic serum) 
 that activates the immune system and initiates an inflammatory response 
 within the joint tissue, which is composed of a lining (L) and sublining (SL) membrane.
  At steady state, the synovial lining membrane is populated by specialized tissue-resident macrophages (TRM) 
  that form a protective barrier and contribute to immune regulation (L-TRM). TRM are also present within the sublining membrane (SL-TRM),
   scattered and fewer in number. 
  Both types of TRM contribute to early immune sensing of environmental cues, exerting overall protective and regulatory roles by limiting excessive inflammation.
   However, during arthritis inflammatory flares, while L-TRM remain preserved, SL-TRM are progressively depleted. 
   This depletion is concomitant with the accumulation of macrophages derived from blood monocytes infiltrating the inflamed joints (MoMac). 
   The loss of the SL-TRM population weakens synovial immune regulation and favors persistent inflammatory cell infiltration through 
   to the local production of cytokines and chemokines. This activation promotes the rapid recruitment of neutrophils (N) from the bloodstream into the joint, 
   where they become the dominant cell population during the acute phase of inflammation.
    Neutrophils amplify the inflammatory response through the release of pro-inflammatory mediators. 
     Based on these biological observations, we modeled the joint inflammatory response as the interaction between an inflammatory signal and three cell populations—N, SL-TRM, and MoMac—capturing recruitment,
      amplification, and regulatory mechanisms in a minimal dynamical framework.

\subsection{Mathematical Modelling.}
The variable $S(t)$ is the inflammatory signal at time $t$. We normalize it by the standard dose of the arthritogenic serum, so that $S(t)$ is dimensionless and represents the relative strength of the inflammatory signal with respect to the standard dose.
We rescale all variables by their respective baseline values in the disease-free state, so that they are dimensionless and represent relative populations with respect to their baseline levels.
Namely variable $N(t)$ (resp. $M(t)$, $T(t)$) designates the relative population of neutrophils (resp. MoMac cells, SL-TRM cells) at time $t$ with respect to the baseline population $\widehat{N}$ (resp. $\widehat{M}$, $\widehat{T}$).\\
The model is given by the following system of ordinary differential equations:
\begin{subequations}\label{eq:system}
\renewcommand{\theequation}{\theparentequation.\arabic{equation}}
\begin{align}
\frac{dS}{dt} &= (a\, N - b\,T) S \label{eq:system:S}\\
\frac{dN}{dt} &= \alpha S N - \delta_N N (N-1) \label{eq:system:N}\\
\frac{dM}{dt} &= \beta S - \delta_M M (M-1) \label{eq:system:M}\\
\frac{dT}{dt} &= -\gamma S T + \delta_M M (1-T) \label{eq:system:T} 
\end{align}
\end{subequations}
where $a,b,\alpha,\beta,\gamma, \delta_N, \delta_M$ are real positive parameters.

Let us explain the meaning of each term in the system~\eqref{eq:system}.
The growth rate of the inflammatory signal $S(t)$ is enhanced by neutrophils $N$ (factor $a$) and inhibited by SL-TRM macrophages $T$ (factor $b$), which translates into the differential equation~\eqref{eq:system:S}.
Equation~\eqref{eq:system:N} describes the dynamics of the neutrophils $N$.
 Neutrophils $N$ are recruited by the inflammatory signal $S$ (factor $\alpha$) and they have a nonlinear logistic self-limitation (factor $\delta_N$). This is modeled by the term $\delta_N N (N-1)$ in the equation for $N$,
which is negative when N > 1 and positive when N < 1, hence in the absence of inflammation it tends to stabilize the
population of neutrophils around its baseline level. 
Equation~\eqref{eq:system:M} describes the dynamics of the MoMac cells $M$. They are recruited from the blood stream at a  rate $\beta S$ proportional to the inflammatory signal. They have a logistic self-regulation mechanism analogous to that of neutrophils.
 Note that the linear rate $\beta S$ of recruitment of MoMac cells is not proportional to the current population of MoMac cells, which models the fact that the recruitment of MoMac cells from the blood stream is not directly influenced by the current population of MoMac cells in the joint, but rather by the level of inflammation.

Equation~\eqref{eq:system:T} describes the dynamics of the SL-TRM cells $T$. The SL-TRM cells $T$ are destroyed by the inflammatory signal at a rate $\gamma S$. When $T$ is less than 1, the term $\delta_M M (1-T)$ is positive and contributes to the growth of $T$, while when $T$ is greater than 1, this term becomes negative and contributes to the decrease of $T$. This models the fact that the population of SL-TRM cells is regulated by the population of MoMac cells, which can promote their proliferation when they are depleted, but can also inhibit them when they are in excess.

In the system~\eqref{eq:system}, we see that inflammation further increases neutrophils (factor $\alpha$), increases MoMac (factor $\beta$), and decreases SL-TRM (factor $\gamma$).
\begin{rem}
The term $\delta_M M (1-T)$ in the equation for $T$ models  the stimulation of  SL-TRM cells by MoMac cells, which is supported by experimental evidence~\cite{STA},\cite{weihaupt26}.
Note that we could have introduced another parameter $\delta_T\neq \delta_M$ in this equation for $T$, but we choose to avoid it, since analysis and numerical simulations showed us that it did not change the dynamics of the crisis.
\end{rem}
To make the model more intuitive, we provide a schematic representation of the interactions between the inflammatory signal and the three cell populations in Figure~\ref{fig:graphe}.
\usetikzlibrary{decorations.markings}
\tikzset{
    mid arrow/.style={
        decoration={markings, mark=at position 0.5 with {\arrow{>}}},
        postaction={decorate}
    },
    mid bar/.style={
        decoration={markings, mark=at position 0.5 with {\arrow{|-}}},
        postaction={decorate}
    }
}
\begin{figure}[htb!]
\begin{center}
    \begin{tikzpicture}[node distance=4cm, auto, >=latex, thick]

        \node[draw, rectangle, fill=red!10, minimum width=3cm, minimum height=1cm] (N) {Neutrophils $N$};
        \node[draw, rectangle, fill=green!10, minimum width=3cm, minimum height=1cm, right=of N] (M) {MoMac $M$};
        \node[draw, rectangle, fill=magenta!10, minimum width=3cm, minimum height=1cm, right=of M] (T) {SL-TRM $T$};
        \node[draw, circle, fill=blue!10, minimum size=2cm, below=of M, yshift=-1cm] (S) {Immune signal $S$};

        \draw[mid arrow, line width=1.5pt] (N) to[bend left=20] node[left, xshift=-2mm, yshift=2mm] {foster $(a)$} (S);
        \draw[dashed,-|, line width=1.5pt] (T) to[bend left=20] node[right, xshift=2mm] {inhibition $(b)$} (S);
        \draw[mid arrow, line width=1.5pt] (S) to[bend left=20] node[left] {activation$(\alpha)$} (N);
        \draw[mid arrow, line width=1.5pt] (S) to[bend left=10] node[left,xshift=8mm, yshift=6mm] {activation $(\beta)$} (M);
        \draw[ dashed, line width=1.5pt, -|] (S) to[bend left=20] node[left, xshift=10mm, yshift=2mm] {inhibition $(\gamma)$} (T);
        \draw[ line width=1.5pt, loop above, looseness=8] (N) to node[above] {regulation $(\delta_N)$} (N);
        \draw[line width=1.5pt, loop above, looseness=8] (M) to node[above] {regulation $(\delta_M)$} (M);
        \draw[mid arrow, line width=1.5pt] (M.east) -- node[below, yshift=-2mm] {stimulation $(\delta_M)$} (T.west);

    \end{tikzpicture}
\end{center}
\caption{Schematic representation of the model}\label{fig:graphe}
\end{figure}
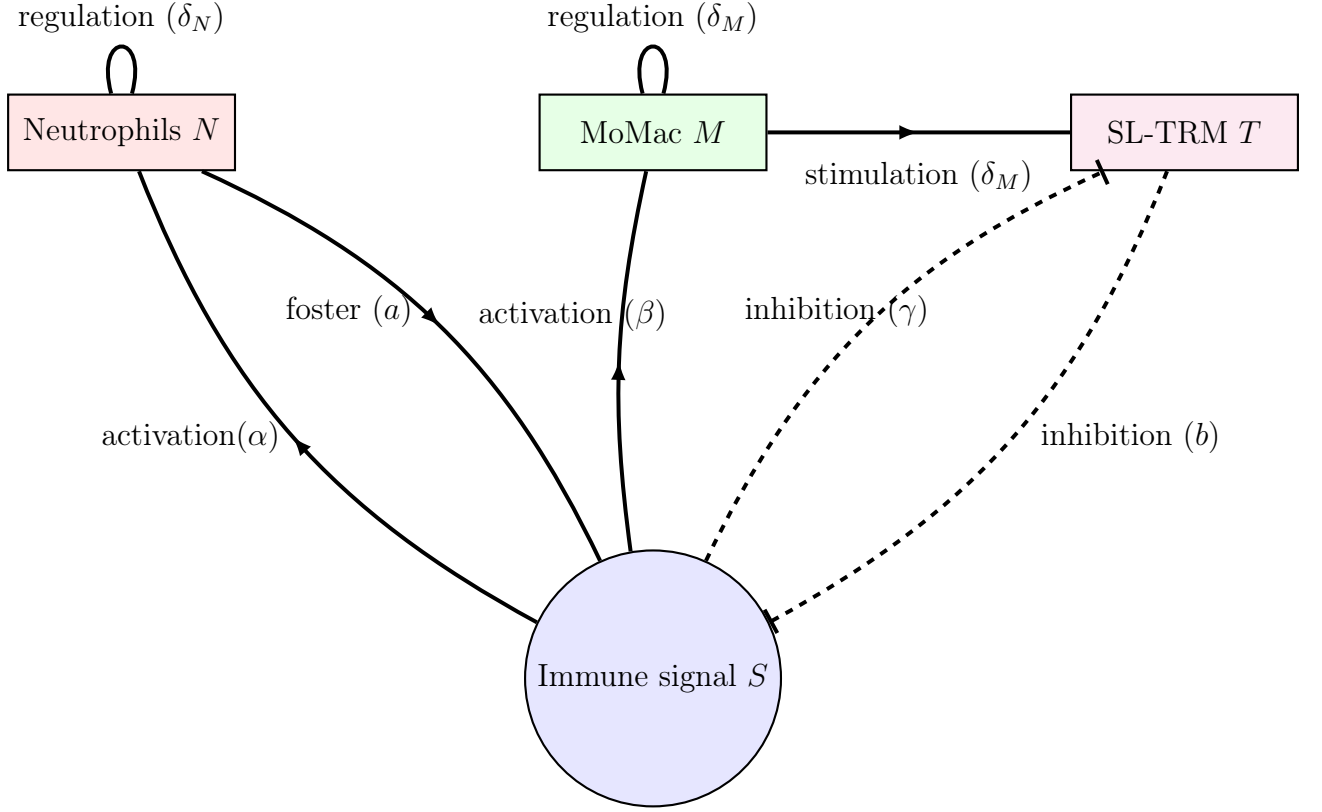

\bigskip


\section{Analysis of the model.}\label{sec:analysis}
\subsection{Existence and non-negativity of the solution.}
Existence and uniqueness of a local solution of the system~\eqref{eq:system} is guaranteed by the Cauchy-Lipschitz theorem since the right-hand side of the system is locally Lipschitz continuous.
\begin{proposition}\label{prop:positivity}
If the initial conditions $S(0)$, $N(0)$, $M(0)$, and $T(0)$ are nonnegative, then the solution $(S(t), N(t), M(t), T(t))$ of the system~\eqref{eq:system} remains nonnegative for all $t \geq 0$.   
\end{proposition}
\begin{proof}
We consider the vector field $F=(f_1, f_2, f_3, f_4)$ for $y=(y_1, y_2, y_3, y_4)\in \mathbb{R}^4$ given by:
\begin{equation*}
  \left\{
\begin{aligned}
f_1(y) &= (a y_2 - b y_4) y_1, \\
f_2(y) &= \alpha y_1 y_2 - \delta_N y_2 (y_2 - 1), \\
f_3(y) &= \beta y_1 - \delta_M y_3 (y_3 - 1), \\
f_4(y) &= -\gamma y_1 y_4 + \delta_M y_3 (1 - y_4).
\end{aligned}
\right.
\end{equation*}
We observe that the vector field $F$ is \emph{quasi-positive}, i.e.\ for each $i=1,2,3,4$, we have $f_i(y) \geq 0$ whenever $y_i = 0$ and $y_j \geq 0$ for all $j \neq i$. Thus from Proposition 2.1 in~\cite{haraux2016}, we deduce that if the initial conditions are nonnegative, then the solution remains nonnegative for all $t \geq 0$.
\end{proof}
This property is biologically meaningful since the variables $S$, $N$, $M$, and $T$ represent cell populations, which cannot be negative. Our model thus ensures that the solutions remain within the biologically relevant domain.       
\subsection{Boundedness of the solution.}
Considering the equation for $T$, one can easily prove that $T(t)$  is always bounded by its initial value $T(0)$ or by 1 if $T(0) \leq 1$, for all $t \geq 0$. 
\begin{rem}
From a biological point of view, the upper bound on T can be interpreted as a modelling assumption reflecting the niche-limited nature of sublining tissue-resident macrophages. Rather than implying that SL-TRM are strictly capped at their disease-free level in all biological situations, this assumption represents the idea that their expansion is constrained by local tissue space and homeostatic regulatory mechanisms.
\end{rem}

\begin{proposition}
Assume that the initial condition $T(0)$ is less than or equal to 1, then
the variable $T(t)$ remains less than or equal to 1 for all $t \geq 0$. 
If $T(0) > 1$, then $T(t)$ remains less than or equal to $T(0)$ for all $t \geq 0$.
\end{proposition}
\begin{proof}
    Let us consider the equation for $T$:
\begin{equation*}
\frac{dT}{dt} = -\gamma S T + \delta_M M (1 - T)
\end{equation*}
Let us recall Gronwall's lemma: if $u(t)$ is a function satisfying the differential inequality
\begin{equation*}
\frac{du}{dt} \leq \alpha(t) u(t)
\end{equation*}
for some function $\alpha(t)$, then we have
\begin{equation*}
u(t) \leq u(0) \exp{\int_0^t \alpha(s) ds}.
\end{equation*}
If $T(0) \leq 1$, we can write the equation for $T-1$ as
\begin{equation*}
\frac{d(T-1)}{dt} = -\gamma S (T-1) - \delta_M M (T-1) -\gamma S 
\leq (-\gamma S - \delta_M M) (T-1)
\end{equation*}
Thus, by Gronwall's lemma, we have
\begin{equation*}
T(t) - 1 \leq (T(0) - 1) \exp{\int_0^t (-\gamma S(s) - \delta_M M(s)) ds} \leq 0
\end{equation*}
which proves that $T(t) \leq 1$ for all $t \geq 0$.\\
In the other case, if $T(0) > 1$, we can write the equation for $T - T(0)$ as
\begin{equation*}
\frac{d(T - T(0))}{dt} = -\gamma S (T - T(0)) - \delta_M M (T - T(0)) + \delta_M M (1 - T(0)) - \gamma S T(0) \leq (-\gamma S - \delta_M M) (T - T(0))
\end{equation*}
Thus, by Gronwall's lemma, we have
\begin{equation*}
T(t) - T(0) \leq (T(0) - T(0)) \exp{\int_0^t (-\gamma S(s) - \delta_M M(s)) ds} = 0
\end{equation*}
which proves that $T(t) \leq T(0)$ for all $t \geq 0$.  
\end{proof}
About the variables $S$, $N$ and $M$, we did not  include degradation terms in the $S$, $N$ and  $M$ equations, which could ensure the boundedness of $S$, $N$ and $M$.
    Indeed we shall see in Section~\ref{sec:numerical_simulations} that in some cases the solution may not be bounded for all time, at least for $S$, $N$ and $M$.
     Actually we are interested in the  immune response to an external stimulus, which can lead to a very high level of inflammation and cell recruitment from the bloodstream or marrow. 
     The transient dynamics of the system can involve large variations in the levels of inflammation and cell populations, which may be shadowed by a model that includes degradation terms that dampens the growth of these variables.
    In our model, the solution may not be bounded for all time,
     but it can still be biologically relevant since we are interested in the early crisis dynamics 
     rather than the long-term behavior of the system.\\
However, we can still derive some lower bounds for the variables $M$ and $N$.    
Using Gronwall's lemma like in the previous proof, we can also bound from below the variables $M$ and  $N$.
\begin{proposition}
Assume that the initial condition $M(0)$ (resp. $N(0)$) is greater than or equal to 1, then
the variable $M(t)$ (resp. $N(t)$) is greater than or equal to 1 for all $t \geq 0$. If $M(0) < 1$ (resp. $N(0) < 1$), then $M(t)$ (resp. $N(t)$) is greater than or equal to $M(0)$ (resp. $N(0)$) for all $t \geq 0$.    
\end{proposition}
\begin{rem}
The lower bounds on M and N should be interpreted in the context of normalized variables and systemic availability.
 They do not imply that monocyte-derived macrophages or neutrophils cannot decrease locally within the joint during resolution, but rather that, in this minimal model,
  their recruitment potential is maintained by circulating and bone marrow reservoirs.

\end{rem}

\subsection{Equilibria points.}
\emph{Equilibria points} of the system are found by setting the right-hand side of the system to zero, i.e.
 the equilibrium points $(S^*, N^*, M^*, T^*)$ satisfy the following system of algebraic equations:
\begin{align}
(a N^* - b T^*) S^* &= 0 \label{eq:equilibrium_S} \\
\alpha S^* N^* - \delta_N N^* (N^* - 1) &= 0 \label{eq:equilibrium_N} \\
\beta S^* - \delta_M M^* (M^* - 1) &= 0 \label{eq:equilibrium_M} \\
-\gamma S^* T^* + \delta_M M^* (1 - T^*) &= 0 \label{eq:equilibrium_T}
\end{align}
\subsubsection{Inflammation-free equilibrium points (S =0).}\label{sec:disease_free}
A first natural equilibrium point is given by $(S^* = 0, N^* = 1, M^* = 1, T^* = 1)$. This point corresponds to the disease-free state, where there is no inflammation ($S^* = 0$) and the populations of neutrophils, MoMac, and SL-TRM are at their baseline    levels ($N^* = 1, M^* = 1, T^* = 1$).

There are other equilibria points with $S^* = 0$  namely
$(S^* = 0, N^* = 0, M^* = 0, T^* = R)$, with arbitrary $R$;
 $(S^* = 0, N^* = 0, M^* = 1, T^* = 1)$;
 $(S^* = 0, N^* = 1, M^* = 0, T^* = R)$, with arbitrary $R$.
\begin{rem}\label{rem:vanishing}
These equilibria points with  $N^* = 0$ or $M^* = 0$ are not biologically relevant since they correspond to the absence of neutrophils or MoMac cells, which is not realistic in a living organism.
Furthermore, we will prove later that these equilibria points are unstable, so they will not be observed in physiological conditions.
\end{rem}

\subsubsection{Inflammatory equilibrium points (S > 0).}\label{sec:inflammatory_equilibria}
We express $N^*$ in terms of $T^*$ from equation~\eqref{eq:equilibrium_S}:
\begin{equation}
N^* = \frac{b}{a} T^*
\label{eq:N_in_terms_of_T}
\end{equation}
Substituting this expression for $N^*$ into equation~\eqref{eq:equilibrium_N}, we obtain:
\begin{equation}
S^* = \frac{\delta_N}{\alpha} \left(\frac{b}{a} T^* - 1\right)
\label{eq:S_in_terms_of_T}
\end{equation}
Since we are now looking for equilibria points with $S^* \neq 0$,~\eqref{eq:equilibrium_M} and~\eqref{eq:equilibrium_T} imply both $M^* \neq 1$ and $T^* \neq 1$.
From equation~\eqref{eq:equilibrium_M} and equation~\eqref{eq:equilibrium_T}, 
we derive  $\delta_M M^* = \frac{\beta S^*}{M^* - 1}$ and $\delta_M M^* = \frac{\gamma S^* T^*}{1 - T^*}$.
Equating these two expressions for $\delta_M M^*$, we obtain:
\[\frac{\beta S^*}{M^* - 1} = \frac{\gamma S^* T^*}{1 - T^*}
\]
Assuming $S^* \neq 0$, we can simplify this equation to:
\[\frac{\beta}{M^* - 1} = \frac{\gamma T^*}{1 - T^*}
\]
From this, we can express $M^*$ in terms of $T^*$:
\begin{equation}
M^* = 1 + \frac{\beta (1 - T^*)}{\gamma T^*} = 1 - \frac{\beta}{\gamma} + \frac{\beta}{\gamma T^*}
\label{eq:M_in_terms_of_T}
\end{equation}  

Substituting the expression for $M^*$  and $S^*$ into the last equation~\eqref{eq:equilibrium_T}, we obtain 
\begin{equation}
  -\gamma \frac{\delta_N}{\alpha} \left(\frac{b}{a} T^* - 1\right) T^* + \delta_M \left(1 - \frac{\beta}{\gamma} + \frac{\beta}{\gamma T^*}\right) (1 - T^*) = 0      
\end{equation}
Since $T^* \neq 0$, we can multiply this equation by $T^*$ and rearrange the terms to obtain a cubic equation for $T^*$:
\begin{equation}
\frac{\gamma}{\alpha} \frac{\delta_N}{\delta_M} \frac{b}{a} {(T^*)}^3 + \left( 1-   \frac{\gamma}{\alpha} \frac{\delta_N}{\delta_M}  - \frac{\beta}{\gamma}\right) {(T^*)}^2 + \left( 2\frac{\beta}{\gamma} -1 \right) T^* - \frac{\beta}{\gamma} = 0
\end{equation}
Rewriting this equation in the standard unitary form ${(T^*)}^3 + p {(T^*)}^2 + q T^* + r = 0$, we have:
\begin{equation}
{\left(T^*\right)}^3 + \frac{\alpha}{\gamma} \frac{\delta_M}{\delta_N} \frac{a}{b} \left(1- \frac{\beta}{\gamma} - \frac{\gamma}{\alpha} \frac{\delta_N}{\delta_M}\right) {(T^*)}^2 + \frac{\alpha}{\gamma} \frac{\delta_M}{\delta_N} \frac{a}{b} \left(2\frac{\beta}{\gamma} -1 \right) T^* - \frac{\alpha}{\gamma} \frac{\delta_M}{\delta_N} \frac{a}{b} \frac{\beta}{\gamma} = 0
\end{equation}
Notice that we need only 3 \emph{non dimensional} parameters $a/b$, $\beta/\gamma$, $\displaystyle{\frac{\gamma \delta_N}{\alpha \delta_M}}$ instead of the 7 original parameters.
In the following, we denote them as $A=a/b$, $B=\beta/\gamma$, and $C=\displaystyle{\frac{\gamma \,\delta_N}{\alpha\,\delta_M}}$.
\begin{rem}
Interestingly, the parameters $A$, $B$, and $C$ have the following biological interpretations:
\begin{itemize}
    \item $A = a/b$. This parameter represents the ratio of the inflammatory effect of the neutrophils with respect to the anti-inflammatory effect of the macrophages, see equation~\eqref{eq:system:S}.
    \item $B = \beta/\gamma$. This parameter represents the ratio of the increasing rate of the MoMac to the decreasing rate of the macrophage cell population, see equations~\eqref{eq:system:M}-\eqref{eq:system:T}.
    \item $C = \frac{\gamma \,\delta_N}{\alpha\,\delta_M} = \left[\frac{\gamma}{\delta_M} : \frac{\alpha}{\delta_N}\right]$. This parameter represents the ratio of the macrophages vs the neutrophiles growth/logistic control parameters in the disease progression, see equations~\eqref{eq:system:N}-\eqref{eq:system:T}.
\end{itemize}   
\end{rem}
Using these parameters, the cubic equation in $T^*$ becomes:
\begin{equation} 
{\left(T^*\right)}^3 +  \frac{A}{C}  \left(1- B- C \right) {(T^*)}^2 + \frac{A}{C} \left(2 B -1 \right) T^* - \frac{A B}{C} = 0
\label{eq:cubic_T}
\end{equation}
\par\noindent We now determine the number of positive roots of the cubic equation~\eqref{eq:cubic_T}.
From basic algebra, we can show that there are at most 3 such intersection points, hence at most 3 equilibria points with $S^* \neq 0$. 
We have to take into account the biological feasibility condition requiring the roots to be real and nonnegative, as stated in Proposition~\ref{prop:positivity}.
Since $\frac{AB}{C} > 0$,  there is at least one positive root. 
But we can be more precise using 
\emph{Descartes' rule of signs}. It states that the number of positive real roots of a polynomial (counting multiplicities) is equal to the number of sign changes between consecutive nonzero coefficients, or less than that by an even number. 
In our case, the coefficients of the cubic polynomial in~\eqref{eq:cubic_T} are:
\[1, \quad \frac{A}{C} \left(1- B- C \right), \quad \frac{A}{C} \left(2 B -1 \right), \quad - \frac{A B}{C}   \]  
So the number of sign changes depends on the signs of the second and third coefficients.
We have the following cases:
\begin{itemize}
\item If $ B +C > 1$ and $B > 1/2$, then there are 3 sign changes and the cubic has one or three positive roots, which can collapse according to their multiplicities.
\item If $B +C  < 1$ and $B > 1/2 $, then there is one sign change and the cubic has only one positive roots.
\item If $B +C  < 1$ and $B < 1/2$, then there is one sign changes and the cubic has only one positive root.
\item If $B +C  > 1$ and $B < 1/2$, then there is one sign change and the cubic has only one positive root.
\end{itemize}

\par\noindent We can use Cardano's method to find the roots of the cubic equation~\eqref{eq:cubic_T}. We rewrite it in the standard form:
\[ T^3 + p T^2 + q T + r = 0
\]
where
\[ p = \frac{A}{C} \left(1- B- C \right), \quad
q = \frac{A}{C} \left(2 B -1 \right), \quad
r = -\frac{A B}{C}
\]
We perform the change of variable $T = x - p/3$ to eliminate the quadratic term and obtain the depressed cubic equation:
\[
x^3 + \mathcal{A} x + \mathcal{B} = 0
\]
where
\[
\mathcal{A} = q - \frac{p^2}{3}, \quad \mathcal{B} = \frac{2 p^3}{27} - \frac{p q}{3} + r
\]
The discriminant of this equation is given by:
\[
\Delta_3 = {\left(\frac{\mathcal{B}}{2}\right)}^2 + {\left(\frac{\mathcal{A}}{3}\right)}^3
\]

The nature of the roots depends on the sign of the discriminant $\Delta_3$:
\begin{itemize}
\item If $\Delta_3 > 0$, there is one real root and two complex conjugate roots.
\item If $\Delta_3 = 0$, all roots are real and at least two are equal.
\item If $\Delta_3 < 0$, all roots are real and unequal.
\end{itemize}
Hence we can characterize the case when we have three distinct real roots by the condition $\Delta_3 < 0$, which is equivalent to:
\[ {\left(\frac{\mathcal{A}}{3}\right)}^3 < - {\left(\frac{\mathcal{B}}{2}\right)}^2.
\]
The parameters $\mathcal{A}$ and $\mathcal{B}$ can be rewritten in terms of the original parameters  using computer algebra as follows:
\[
\mathcal{A}=\frac{A}{ 3 C^2} \Big(-A {(B + C - 1)}^2 + 3 C(2B - 1)\Big)
\]
\[
\mathcal{B}= \frac{A} {27 C^3} \Big(-2A^2 {(B + C - 1)}^3 + 9 A C (2 B - 1)(B + C - 1) - 27 B C^2\Big)
\]
The discriminant $\Delta_3$ can be rewritten in terms of the original parameters as follows:
\[
\Delta_3 = \frac{A^2} {108 C^4} \Bigg(A^2 (4BC-1) {(B + C - 1)}^2 - 2AC (2B-1)(B^2-9BC-B-2) + 27 B^2 C^2\Bigg)
\]
\par\noindent

The real roots can be found using Cardano's formulas. 
If $\Delta_3 > 0$, the unique real root is given by:
\[
x^* = \sqrt[3]{-\frac{\mathcal{B}}{2} + \sqrt{\Delta_3}} + \sqrt[3]{-\frac{\mathcal{B}}{2} - \sqrt{\Delta_3}}.
\]
If $\Delta_3 < 0$, the three real roots are given by:
\[
x_k^* = 2 \sqrt{-\frac{\mathcal{A}}{3}} \cos\left(\frac{1}{3} \arccos\left(\frac{3 \mathcal{B}}{2 \mathcal{A}} \sqrt{-\frac{3}{\mathcal{A}}}\right) + \frac{2 k \pi}{3}\right), \quad k = 0, 1, 2.
\]
If $\Delta_3 = 0$, the real roots are given by:
\[x_1^* = 2 \sqrt[3]{-\frac{\mathcal{B}}{2}}, \quad x_2^* = x_3^* =  \sqrt[3]{\frac{\mathcal{B}}{2}}.
\]
\par\noindent
Here we plot a case with three distinct positive roots (see Figure~\ref{fig:three_roots}).
Namely we take $A = 1.1$, $B = 5$ and $C = \frac{2}{3}$ which gives $\Delta_3 \approx -0.49$.
\begin{figure}[htbp]
    \centering
    \includegraphics[width=0.6\textwidth]{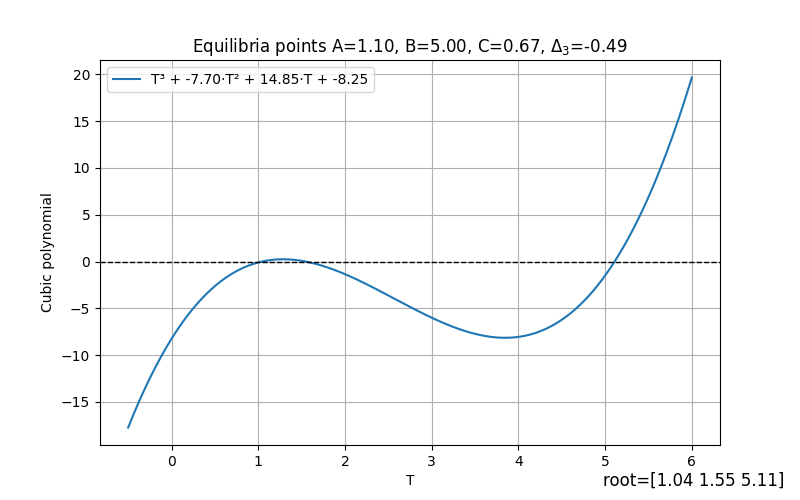}
    \caption{Cubic equation with three distinct real roots.}\label{fig:three_roots}
\end{figure}
\par\noindent 
Finally, we can recover the roots $T^*$ of the original cubic equation 
using the relation $T^* = x^* - p/3$.
\bigskip\\
We have to make sure that the nonegative roots $T^*$ found are biologically feasible, i.e.\ that they lead to nonnegative values of $S^*$, $N^*$, and $M^*$.
For instance in the case $A=1.1$, $B=5$, $C=2/3$ plotted in Figure~\ref{fig:three_roots}, we have three distinct positive roots of the cubic equation~\eqref{eq:cubic_T},
but none satisfying the biological feasibility conditions.\smallskip\\
From equations~\eqref{eq:N_in_terms_of_T},~\eqref{eq:S_in_terms_of_T}, and~\eqref{eq:M_in_terms_of_T}, we have: 
\[N^* = \frac{b}{a} T^*, \quad S^* = \frac{\delta_N}{\alpha} \left(\frac{b}{a} T^* - 1\right), \quad M^* = 1 - \frac{\beta}{\gamma} + \frac{\beta}{\gamma T^*}
\]
Thus, the conditions for biological feasibility are:
\begin{itemize}
\item $S^* \geq 0$ iff $T^* \geq \frac{a}{b} = A$  
\item $N^* \geq 0$  iff $T^* \geq 0$ (always true since we consider only positive roots)
\item $M^* \geq 0$ is always true if $B \leq 1$; otherwise if $B > 1$ it holds iff $T^* \leq  \frac{B}{B-1}$.
\end{itemize}
So the biological feasibility conditions can be summarized as follows:
\begin{align}
&A  \leq  T^*   &\text{if } B \leq 1, \label{feas_T1} \\
&A  \leq  T^*   \leq  \frac{B}{B-1}&  \text{if } B > 1. \label{feas_T2}
\end{align}
\begin{rem}
In the case $B>1$, if $A > \frac{B}{B-1}$, then     the feasibility condition~\eqref{feas_T2} cannot be satisfied and there is no feasible root. In the limit case $A = \frac{B}{B-1}$, the only feasible root is $T^* = A = \frac{B}{B-1}$
 which corresponds to $S^* = 0$, $N^* = 1$, $M^* = 0$ which has already been found in the inflammation-free section~\ref{sec:disease_free}. In the rest of this section, we will assume that $A < \frac{B}{B-1}$ if $B > 1$.
\end{rem}
Let us consider first the particular case $A = a/b = 1$. In this case, we have an obvious solution $T^* = 1$ which corresponds to the disease-free equilibrium point.
Therefore we can factor the cubic polynomial as follows:
    \begin{equation}
    (T^* - 1) \left( {(T^*)}^2 - \frac{1}{C} (B - 1) T^* + \frac{B}{C} \right) = 0.
    \label{eq:factored_cubic}
    \end{equation}
 If $B \leq 1$, all the coefficient of the quadratic polynomial are positive, so there is only one equilibrium, the disease-free $(S^*, N^*, M^*, T^*) = (0, 1, 1, 1)$. 
 Otherwise if $B>1$ the feasibility condition~\eqref{feas_T2} reads:
\begin{equation}
1 \leq T^* \leq \frac{B}{B-1}.
\end{equation}
 We can check easily that only the disease-free root $T^* = 1$ satisfies this condition, while the other two roots are outside this interval.
 Indeed if the discriminant of the quadratic factor $\Delta = \frac{{(B-1)}^2 - 4BC}{C^2}$ is non-negative, we have $C \leq \frac{{(B-1)}^2}{4B}$, 
 hence $\frac{B-1}{2C} \geq  \frac{2B}{B-1}$ thus the root $T_+^* = \frac{B-1}{2C} + \frac{\sqrt{\Delta}}{2} \geq \frac{2B}{B-1} > \frac{B}{B-1}  $ is not feasible.
 Consider the other root $T_-^* = \frac{B-1}{2C} - \frac{\sqrt{\Delta}}{2}$.
 Should this root be feasible, we would have $ T_-^* \leq \frac{B}{B-1}$. Since $T_+^* > \frac{B}{B-1}$, the quadratic factor would have to be negative or zero at $T^* = \frac{B}{B-1}$. But we have: ${\left(\frac{B}{B-1}\right)}^2 - \frac{1}{C} (B - 1) \frac{B}{B-1} + \frac{B}{C} = \frac{B^2}{{(B-1)}^2} > 0$ so the root $T_-^*$ is not feasible either.  
\bigskip\\
In the general case $A \neq 1$, the feasibility conditions~\eqref{feas_T1}-\eqref{feas_T2} have to be checked for each positive root found. 
For $A>1$, we will show that none of the positive roots satisfy the feasibility conditions (like in figure~\ref{fig:three_roots}) whereas for $A<1$, there is always exactly one feasible root (in addition to the disease-free equilibrium). 
To this aim, we consider four cases  depending on the signs of $A - 1$ and $B - 1$.
\begin{description}
\item[-] Case $A > 1$ and $B \geq 1$.\\ 
Let us define $f(T^*)$ as the cubic polynomial in the left-hand side of equation~\eqref{eq:cubic_T}. 
If $B>1$ we have
\begin{eqnarray*}
f(T^*)&=&{T^*}^3-\frac{A}{C}(B-1){T^*}^2-A{T^*}^2+\frac{A}{C}(B-1)T^*+\frac{A}{C}BT^*-\frac{AB}{C}\\\nonumber
&=&\frac{A}{C}(B-1)(T^*-\frac{B}{B-1})(1-T^*)+{T^*}^2(T^*-A)
\end{eqnarray*} 
So $f(T^*)>0$, for $T^* \in [A, \frac{B}{B-1}]$  if $B>1$.\\
If $B=1$ then $f(T^*)={T^*}^2(T^*-A)+\frac{A}{C}(T^*-1)>0$ for $T^*>A$.\\
Hence, thanks to~\eqref{feas_T1} and~\eqref{feas_T2}, there is no feasible root in this case.
\item[-] Case $A > 1$ and $B < 1$.
Since \begin{eqnarray*}
f(T^*)=\frac{A}{C}(B-1)(T^*-\frac{B}{B-1})(1-T^*)+{T^*}^2(T^*-A)&>&0
\end{eqnarray*}
for $T^* \geq A$ and $B<1$, there is no feasible root in this case either. 
\item[-] Case $A < 1$ and $B  \geq1$.\\
We compute $f(1)=1-A>0$ and $f(A)=(A-1)\frac{A}{C}(A(1-B)+B)<0$, so there is at least one feasible root in this case.\\
If $B>1$, computing the derivative $f'(T^*)$ we have:
\begin{eqnarray*}
f'(T^*) &=& \frac{A}{C}(B-1)\big[(\frac{B}{B-1}-T^*)+(1-T^*)\big]+2{T^*}(T^*-A)+{T^*}^2 >0
\end{eqnarray*}
for $T^* \in [A,1]$. So there is exactly one root in the interval $[A,1]$.
Since once again $f(T^*)=\frac{A}{C}(B-1)(T^*-\frac{B}{B-1})(1-T^*)+{T^*}^2(T^*-A)$ is positive for $T^* \in [1, \frac{B}{B-1}]$, there is exactly one feasible root in this case.\\
If $B=1$, we have $f(T^*)={T^*}^2(T^*-A)+\frac{A}{C}(T^*-1)>0$ for $T^*>1$ and $f'(T^*)={T^*}^2+2{T^*}({T^*}-A)+\frac{A}{C}>0$ for $T^* \in [A,1]$. 
So there is exactly one feasible root in this case too.
\item[-] Case $A < 1$ and $B < 1$.\\
Since $f(1)=1-A>0$ and $f(A)=(A-1)\frac{A}{C}(A(1-B)+B)<0$, there is at least one feasible root in this case too.
Let $\tilde{T}=T^*-A$ and $\tilde{f}(\tilde{T})=f(\tilde{T}+A)$. Then the equation $\tilde{f}(\tilde{T})=0$ is written as:
\begin{eqnarray*}
\tilde{T}^3 +\left(3A+\frac{A}{C}(1-B-C)\right)\tilde{T}^2&+&\left(3A^2+2\frac{A^2}{C}(1-B-C)+\frac{A}{C}(2B-1)\right)\tilde{T}\\
&+&\frac{A}{C}(A-1)\left(A(1-B)+B\right)\\
    =\tilde{T}^3+\left(2A+\frac{A}{C}(1-B)\right)\tilde{T}^2&+&\left(A^2+2\frac{A^2}{C}(1-B)+\frac{A}{C}(2B-1)\right)\tilde{T}\\
    &+&\frac{A}{C}(A-1)\left(A(1-B)+B\right)=0.
    \end{eqnarray*}
    Using \textsl{Descartes' rule of signs} we can see that there is exactly one positive root $\tilde{T}$ 
    of the cubic equation $\tilde{f}(\tilde{T})=0$  as the coefficients of $\tilde{T}^3$ and $\tilde{T}^2$ are positive and the constant term is negative. 
    So there is exactly one feasible root $T^*\geq A$ in this case too.
\end{description}
\par\noindent Let us summarize. The feasible equilibria points of the system can be classified as follows:
\begin{itemize}
\item Disease-free equilibrium: $(S^*, N^*, M^*, T^*) = (0, 1, 1, 1)$.
\item Other equilibria with $S^* = 0$: 
$S^* = 0, N^* = 0, M^* = 0, T^* = R$, with arbitrary $R$;
 $S^* = 0, N^* = 0, M^* = 1, T^* = 1$;
 $S^* = 0, N^* = 1, M^* = 0, T^* = R$, with arbitrary $R$.
\item Feasible equilibria with $S^* \neq 0$: 
None when $A \geq 1$. We shall see in the numerical simulations that in this case the system quickly diverges away from the disease-free equilibrium, no matter how weak is the initial inflammation. This suggests that the system is in a persistent inflammatory state with $S^* \neq 0$.
When $A < 1$, there is only one biologically feasible equilibria with $S^* \neq 0$. We shall see in the next section that this equilibrium is unstable, and the numerical simulations suggest that the system converges to the disease-free equilibrium, hence the inflammation is resolved.
\end{itemize} 
\subsection{Stability analysis.}
\emph{Linear stability analysis} of the equilibrium point $ (S^*, N^*, M^*, T^*)$ is performed by linearizing the differential equations system around this point.
The Jacobian matrix of the system is given by:
\[
J = \begin{pmatrix}
aN - bT & aS & 0 & -bS \\
\alpha N & \alpha S - \delta_N (2N-1) & 0 & 0 \\
\beta & 0 & -\delta_M (2M-1) & 0 \\
-\gamma T & 0 & \delta_M (1-T) & -\gamma S - \delta_M M
\end{pmatrix}
\]
According to the linear stability theory, the equilibrium point is (locally asymptotically) stable if all eigenvalues of the Jacobian matrix have negative real parts, and unstable if at least one eigenvalue has a positive real part. In what follows,
 we analyze the stability of the different equilibria points found in the previous section. By locally asymptotically stable, we mean that if the system is perturbed slightly from the equilibrium point, it will return to the equilibrium point as time goes to infinity. On the other hand, if the equilibrium point is unstable, a small perturbation will cause the system to move away from the equilibrium point.
\subsubsection{Stability of  inflammation free equilibria (S =0).}\label{sec:stab}
We analyze first the stability of the disease-free equilibrium point $(S^*, N^*, M^*, T^*) = (0, 1, 1, 1)$.
At this point, the Jacobian matrix simplifies to:
\[
J|_{(S^*, N^*, M^*, T^*) = (0, 1, 1, 1)} =
\begin{pmatrix}
a - b & 0 & 0 & 0 \\
\alpha & - \delta_N & 0 & 0 \\
\beta & 0 & -\delta_M & 0 \\
-\gamma & 0 & 0 & -\delta_M
\end{pmatrix}
\]
We notice immediately that the eigenvalues of this matrix are  $a-b$, $-\delta_N$ and $-\delta_M$.
 The last two are negative since $\delta_N, \delta_M > 0$. The remaining eigenvalue  is $a-b$. Thus the disease-free equilibrium is stable if $A = a/b < 1$ and unstable if $A =a/b > 1$. 
This means that if the inflammatory effect of Neutrophiles cells controlled by paramater $a$ is less than the curing effect of the  SL-TRM macrophages controlled by parameter $b$, the system will return to the disease-free state.
 Conversely, if $A=a/b > 1$, the system will move away from this equilibrium, potentially leading to a persistent inflammatory state.
The particular case $A=1$ corresponds to a bifurcation point where the stability of the disease-free equilibrium changes. In this case, a more detailed analysis is needed to determine the behavior of the system near this point, which  involve higher-order terms in the Taylor expansion of the system around the equilibrium.
The numerical simulations, see section~\ref{sec:numerical_simulations}, suggest that for $A=1$ the system quickly diverges from the disease-free equilibrium. 
Indeed in this case the disease-free equilibrium is unstable, as we prove below.
\par For $a=b$ the disease-free equilibrium is a non-hyperbolic equilibrium with eigenvalues $0$, $-\delta_N$, $-\delta_M$ and $-\delta_M$. 
The stability of this equilibrium can be analyzed using the center manifold theory, see~\cite{perko}. We prove here that the center manifold is repulsive for $S^* \geq 0$, hence the equilibrium is unstable.\\
In this case there is one center direction and three stable directions. We define new variables $x=S$, $y=N-1$, $z=M-1$, and $w=T-1$ to shift the equilibrium point to the origin. The system can be rewritten as:
\begin{eqnarray*}
\dot{x} &=&  ax (y - w) \\
\dot{y} &=& (y+1)(\alpha x - \delta_N y ) \\
\dot{z} &=& \beta x - \delta_M z (z+1)\\
\dot{w} &=& -\gamma x(w + 1) - \delta_M w (z+ 1)
\end{eqnarray*} 
The linear part of the system is given by:
\begin{eqnarray*}
\dot{x} &=& 0 \\
\dot{y} &=&\alpha x - \delta_N y \\
\dot{z} &=&\beta x - \delta_M z \\
\dot{w} &=&-\gamma x - \delta_M w
\end{eqnarray*}
where the center direction is given by the variable $x$ and the stable directions are given by the variables $y$, $z$, and $w$.
The center manifold can be expressed as a graph of the form $h(x) = (h_1(x), h_2(x), h_3(x))$ where $h_i(0) = 0$ for $i=1,2,3$.
We can compute the Taylor expansion of $h$ and find that the first nonzero term is given by $h_1(x) = \frac{\alpha}{\delta_N} x + o(x)$, $h_2(x) = \frac{\beta}{\delta_M} x + o(x)$, and $h_3(x) = -\frac{\gamma}{\delta_M} x + o(x)$.
The dynamics on the center manifold is given by: 
\begin{eqnarray*}
\dot{x} &=& ax^2\left(\frac{\alpha}{\delta_N}+ \frac{\gamma}{\delta_M}\right) + o(x^2)
\end{eqnarray*}
Since $\frac{\alpha}{\delta_N}+ \frac{\gamma}{\delta_M} > 0$, the center manifold is repulsive for $x \geq 0$, hence the equilibrium is unstable.

The stability of other equilibria points can be analyzed similarly by evaluating the Jacobian matrix at those points and inspecting the eigenvalues.
\begin{itemize}
\item For $(S^* = 0, N^* = 0, M^* = 0, T^* = R)$, with arbitrary $R$, the Jacobian matrix becomes:
\[
J|_{(S^*, N^*, M^*, T^*) = (0, 0, 0, R)} =
\begin{pmatrix}
- b R & 0 & 0 & 0 \\
0 & \delta_N & 0 & 0 \\
\beta & 0 & \delta_M & 0 \\
-\gamma R & 0 & \delta_M (1 - R) & 0
\end{pmatrix}
\]
The eigenvalues of this matrix are $-b R$, $\delta_N$, $\delta_M$, and $0$. Since $\delta_N$ and $\delta_M$ are positive, this equilibrium point is unstable. Even if the immune signal will decrease exponentially fast, the populations of neutrophiles and MoMac will grow out of bound. This unstable equilibrium where the population of neutrophiles and MoMac is zero is  not physiologically relevant and will not be observed in practice, see Remark~\ref{rem:vanishing}.
\item For $(S^* = 0, N^* = 0, M^* = 1, T^* = 1)$, the Jacobian matrix is:
\[
J|_{(S^*, N^*, M^*, T^*) = (0, 0, 1, 1)} =
\begin{pmatrix}
-b & 0 & 0 & 0 \\
0 & \delta_N & 0 & 0 \\
\beta & 0 & -\delta_M & 0 \\
-\gamma & 0 & 0 & -\delta_M
\end{pmatrix}
\]
The eigenvalues of this matrix are $-b$, $\delta_N$, $-\delta_M$, and $-\gamma$. Since there is a positive eigenvalue, this equilibrium point is unstable. Neutrophiles can grow out of bound. This equilibrium where the population of neutrophiles is zero will not be observed in practice either, see Remark~\ref{rem:vanishing}.
\item For $(S^* = 0, N^* = 1, M^* = 0, T^* = R)$, with arbitrary $R$, the Jacobian matrix is
\[
J|_{(S^*, N^*, M^*, T^*) = (0, 1, 0, R)} =
\begin{pmatrix}
a - b R & 0 & 0 & 0 \\
\alpha & -\delta_N & 0 & 0 \\
\beta & 0 & \delta_M & 0 \\
-\gamma R & 0 & \delta_M (1 - R) & 0
\end{pmatrix}
\]
There is a positive eigenvalue $\delta_M$, hence this equilibrium point is unstable. MoMac can grow out of bound. This equilibrium  where the population of MoMac is zero will not be observed in practice either, see Remark~\ref{rem:vanishing}.
\end{itemize}
\subsubsection{Stability of inflammatory equilibrium points (S> 0).}
The stability of the equilibria with $S^* \neq 0$ can be analyzed similarly by evaluating the Jacobian matrix at those points and determining the eigenvalues.
Taking into account the equilibrium equations~\eqref{eq:equilibrium_S},~\eqref{eq:equilibrium_N}, and~\eqref{eq:equilibrium_T}, the Jacobian matrix at $(S^*, N^*, M^*, T^*)$ is:

\[
J = \begin{pmatrix}
0 & aS & 0 & -bS \\
\alpha N &  - \delta_N (N) & 0 & 0 \\
\beta & 0 & -\delta_M (2M-1) & 0 \\
-\gamma T & 0 & \delta_M (1-T) & - \delta_M\frac{M}{T} 
\end{pmatrix}
\]
The eigenvalues of this matrix can be found by solving the characteristic polynomial $\det(J - \lambda I) = 0$. A straightforward computation shows that the characteristic polynomial is given by:
\begin{eqnarray*}
&P(\lambda) = \lambda^4 + \Big(\delta_N N + \delta_M (2M-1) + \delta_M \frac{M}{T}\Big) \lambda^3 \\
&+ \Big(\delta_N N \delta_M(2M-1 + \frac{M}{T}) + {\delta_M}^2 (2M-1)\frac{M}{T} -a S \alpha N - b S \beta \gamma \Big) \lambda^2 \\
&+ \Big(- a S \alpha N \delta_M (2M-1+\frac{M}{T})  - b S\gamma T\delta_N N + \delta_N N {\delta_M}^2(2M-1)\frac{M}{T}- b{\delta_M}^2 M^2(1-T)\Big) \lambda \\
&- a S \alpha N {\delta_M}^2 (2M-1) \frac{M}{T} - b \delta_N N {\delta_M}^2 M^2 (1-T)
\end{eqnarray*}
Due to the Descartes' rule of signs, we claim that the polynomial $P(\lambda)$ has one positive root (hence the equilibrium is unstable). Indeed, the coefficients of $P(\lambda)$ are positive for the terms in 
$\lambda^4$ and $\lambda^3$, whereas the constant term is negative. To prove that there is only one change of sign and hence one positive root, we have to check that if the coefficient of $\lambda$ is positive then the coefficient of 
$\lambda^2$ is positive too. 

\noindent
If the coefficient of $\lambda$ is positive we have the following inequality:
\[
 a S \alpha N \delta_M (2M-1+\frac{M}{T})+ b{\delta_M}^2 M^2(1-T)+ b S\gamma T\delta_N N < \delta_N N {\delta_M}^2(2M-1)\frac{M}{T}
\]
Thanks to the positivity of all terms, we can deduce the following inequalities:
\begin{equation}
 a S \alpha N \delta_M (2M-1+\frac{M}{T}) < \delta_N N {\delta_M}^2(2M-1)\frac{M}{T},
\label{ineq1}
\end{equation}
and
\begin{equation}
 b S\gamma T < {\delta_M}^2(2M-1)\frac{M}{T}.
\label{ineq2}
\end{equation}
Inequality~\eqref{ineq1} implies that 
\begin{equation}\label{ineq3}
a S \alpha N < a S \alpha N\left(\frac{{\delta_M}^2 {(2M-1+\frac{M}{T}) }^2}{{\delta_M}^2(2M-1)\frac{M}{T}}\right)<\delta_N N \delta_M (2M-1+\frac{M}{T}).
\end{equation}
Due to~\eqref{ineq2} and~\eqref{ineq3}, the coefficient of $\lambda^2$ is positive.
\par\noindent
This proves that biologically feasible equilibrium with $S^* \neq 0$, which only exists if $A<1$ see section~\ref{sec:inflammatory_equilibria},  is unstable.
For $A\geq 1$ the only biologically feasible equilibrium is the disease-free equilibrium which is unstable, as we have seen in section~\ref{sec:stab}.

\subsection{Bifurcation diagram.}
We perform here a bifurcation analysis with respect to the \emph{non dimensional} parameters $A=a/b$. Naturally, we restrict our analysis to the biologically relevant feasible equilibria.
From the stability analysis above, we see that we  have only one  bifurcation at  $A=1$: for $A<1$ there are two biologically feasible equilibria, one  inflammatory one with $S^* \neq 0$ which is unstable and the disease-free equilibrium which  is stable; 
for $A\geq 1$ the disease-free equilibrium is unstable and there is no other biologically feasible equilibrium.
The number and stability of equilibria change as the parameter $A$ crosses the critical value $A=1$, indicating  a bifurcation point in the system's dynamics.   
We can plot the bifurcation diagram in the $(A,S^*)$ plane (see Figure~\ref{fig:bifurcation}).
\begin{figure}
 \includegraphics[scale=0.5]{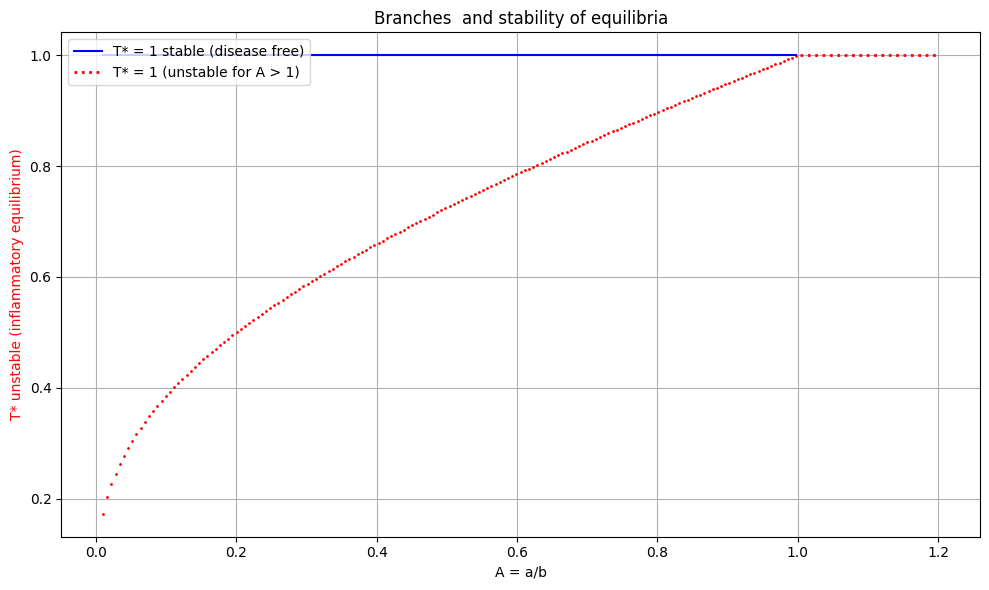}      
 \caption{Bifurcation diagram in the $(A,T^*)$ plane.}\label{fig:bifurcation}
\end{figure}
\section{Numerical simulations.}\label{sec:numerical_simulations}
In this section we present some numerical simulations of the model to illustrate the different dynamic regimes and the effect of the parameter $A=a/b$, the initial immune signal $S(0)$, as well as the initial value of SL-TRM $T(0)$ on the healing prognosis.
Let us emphasize that these are qualitative scenarios illustrating dynamic regimes, and not a quantitative fit to experimental kinetics.

Below we reproduce some typical simulations of the model using the following parameters:
$\alpha=1,\, \beta=1,\, \gamma=1,\,
       \delta_N=1,\, \delta_M=1,\, a=0.35,\; b=1$,  
       duration=20 days and initial conditions $S(0)=1,\, N(0)=1,\, M(0)=1,\, T(0)=1$.
        
\begin{figure}[htbp]
    \centering
    \includegraphics[width=.7\textwidth]{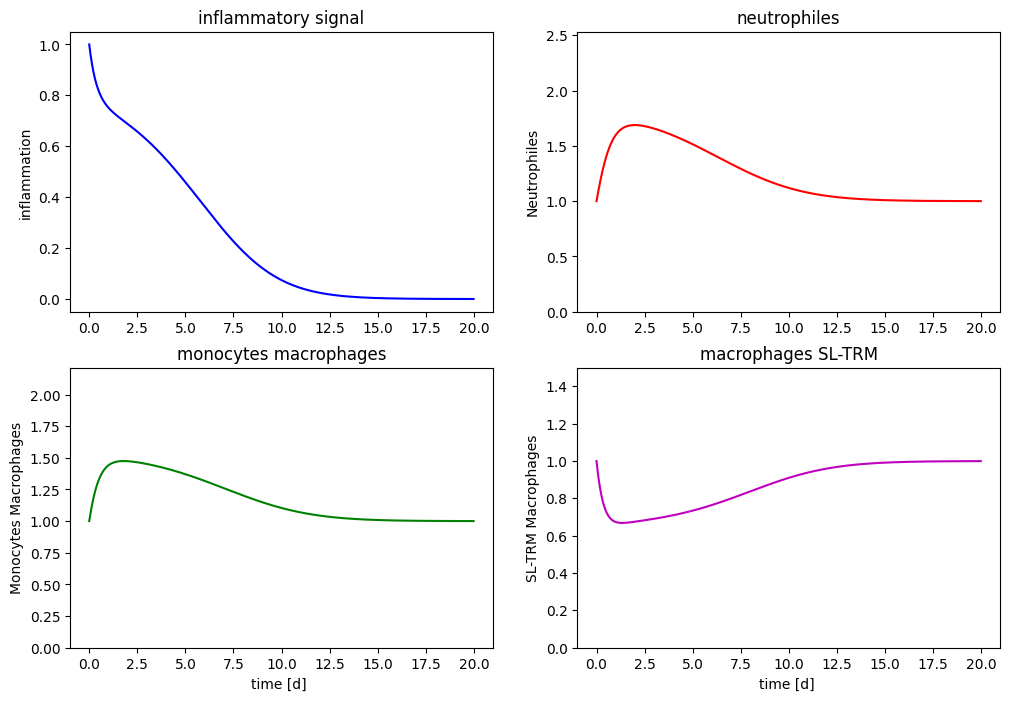}
    \caption{$A=0.35$}\label{035}
\end{figure}
First let us take $a=0.38$, keeping the other parameters unchanged.
\begin{figure}[htbp]
    \centering
    \includegraphics[width=.7\textwidth]{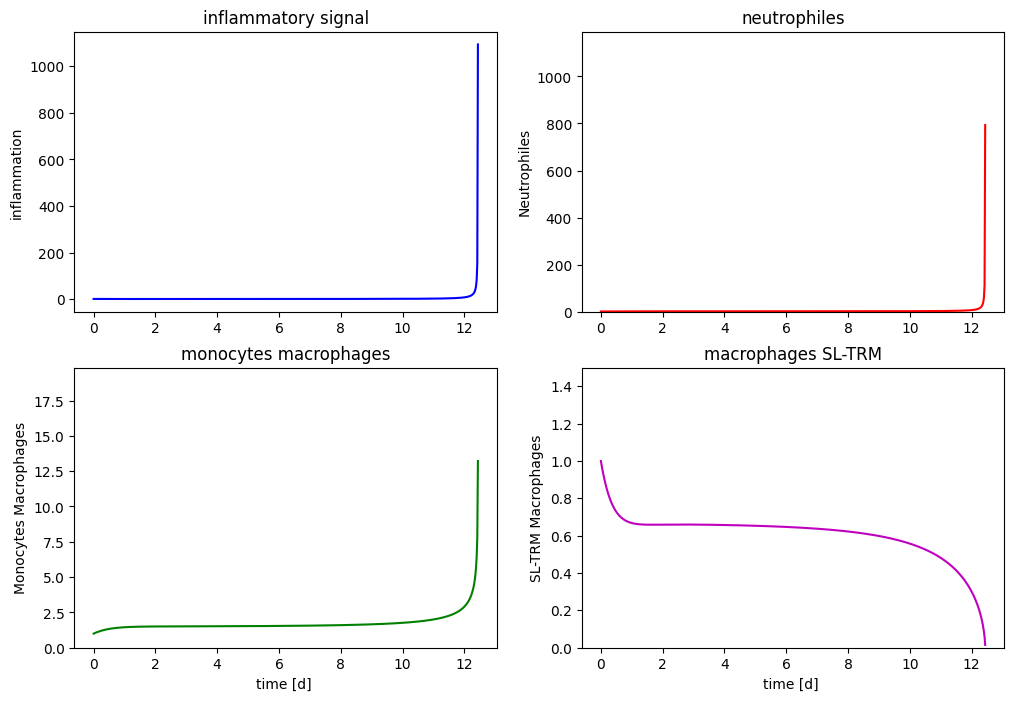}
    \caption{$A=0.38$}\label{038}
\end{figure}
We see clearly  in Fig.~\ref{035} that for $A=0.35 < 1$ the system converges to the disease-free equilibrium, whereas for $A=0.38$ the system diverges away from this equilibrium, indicating a persistent inflammatory state, see Fig.~\ref{038}. 
We can also decrease the initial immune signal $S(0)$ to see if the system converges to the disease-free equilibrium or diverges away from it.
Halving the initial immune signal i.e.~taking $S(0)=0.5$ and $A=0.38$, we obtain the following result:
\begin{figure}[htbp]
    \centering
    \includegraphics[width=.7\textwidth]{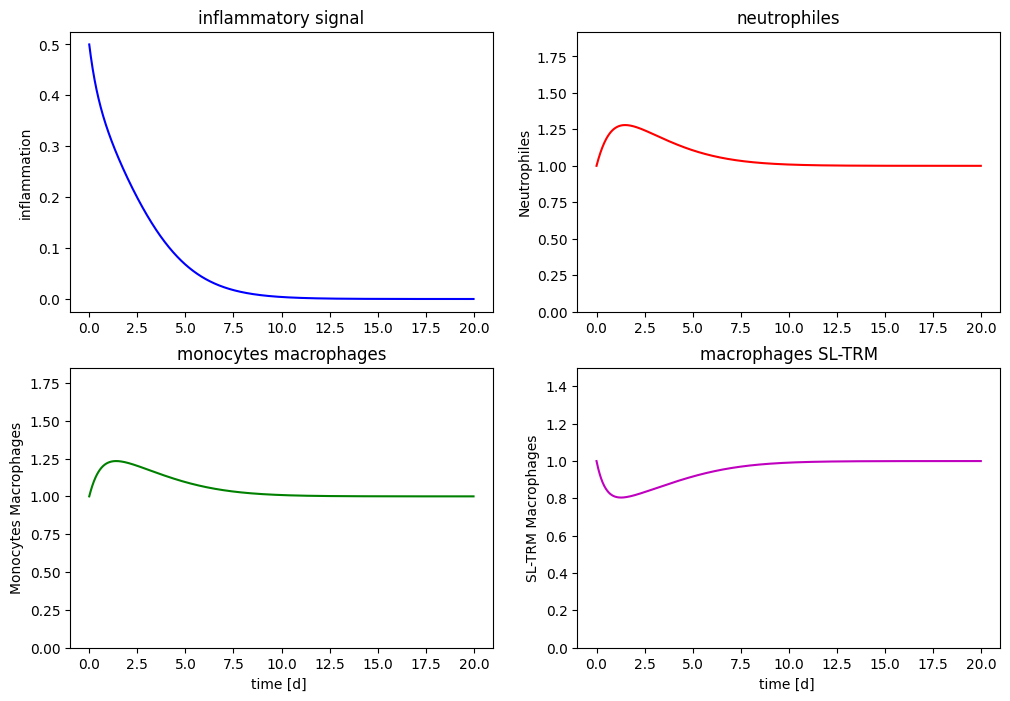}
    \caption{$A=0.38, S(0)=0.5$}\label{038_05}
\end{figure}
We see in Fig.~\ref{038_05} that even if $A=0.38$, the system converges to the disease-free equilibrium, indicating that the initial immune signal is not strong enough to trigger a persistent inflammatory state.
Keeping the same halved initial immune signal, $S(0)=0.5$, 
and the same $N(0)=1$ and $A=0.38$, but starting from an immunocompromised state, that is $M(0)=0.1,\, T(0)=0.05$, we now observe in Fig.~\ref{038_depressed} that the system diverges away from the disease-free equilibrium, indicating a persistent inflammatory state, even if the initial immune signal is low.
\begin{figure}[htbp]
    \centering
    \includegraphics[width=.7\textwidth]{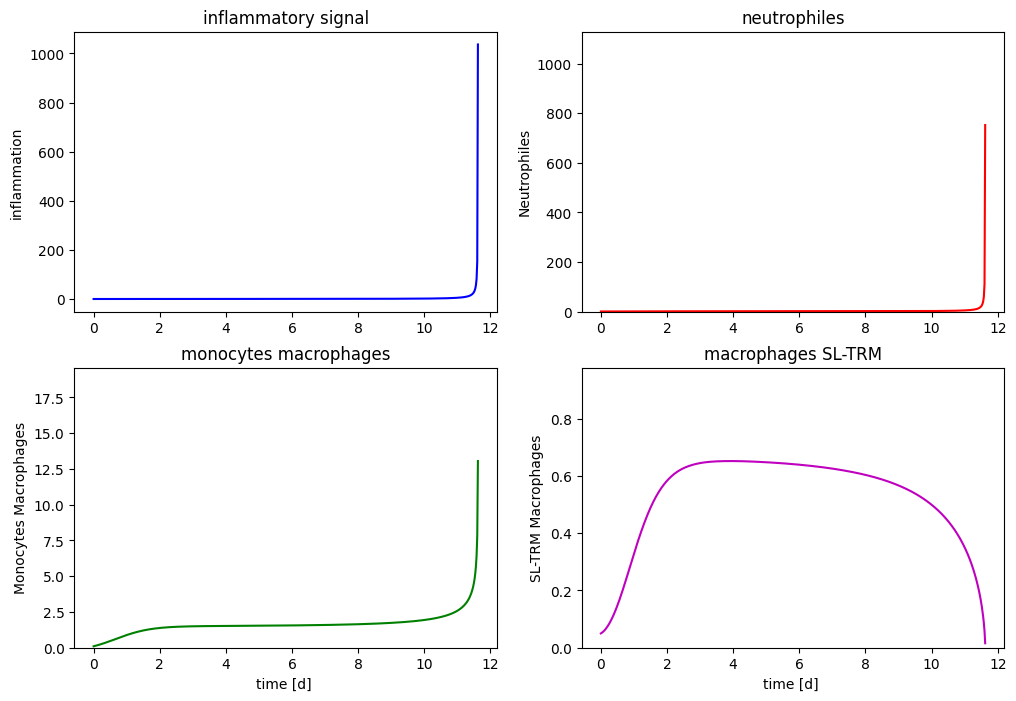}
    \caption{$A=0.38$, immunocompromised initial conditions}\label{038_depressed}
\end{figure}
\begin{figure}[htbp]
    \centering
    \includegraphics[width=.7\textwidth]{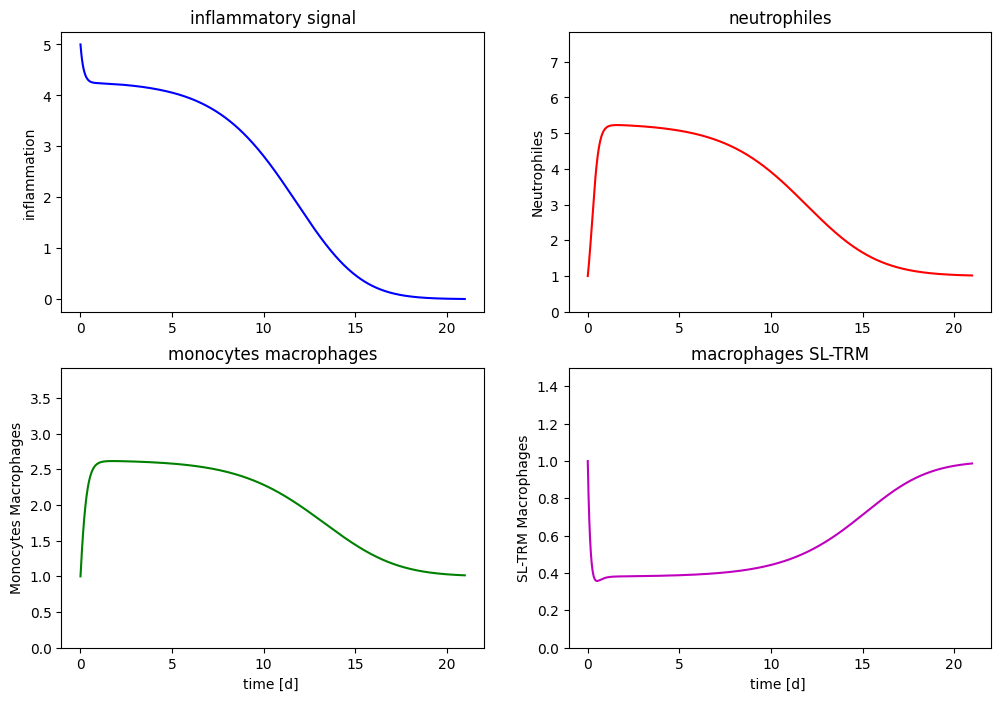}
    \caption{$A=0.072, S(0)=5$}\label{A=0,072}
\end{figure}
We can also start from a very high initial immune signal, taking $S(0)=5,\, N(0)=1, \,M(0)=1,\, T(0)=1$ and $A=0.072$. We observe in Fig.~\ref{A=0,072} that the system returns to the disease-free equilibrium, indicating that even if the initial immune signal is high, the system can still heal when $A=0.072 \ll 1$.
But if we slightly increase $A$ to $A=0.073$, keeping the other parameters unchanged, we observe in Fig.~\ref{A=0,073} that the system diverges away from the disease-free equilibrium, indicating a persistent inflammatory state.
\begin{figure}[htbp]
    \centering
    \includegraphics[width=.7\textwidth]{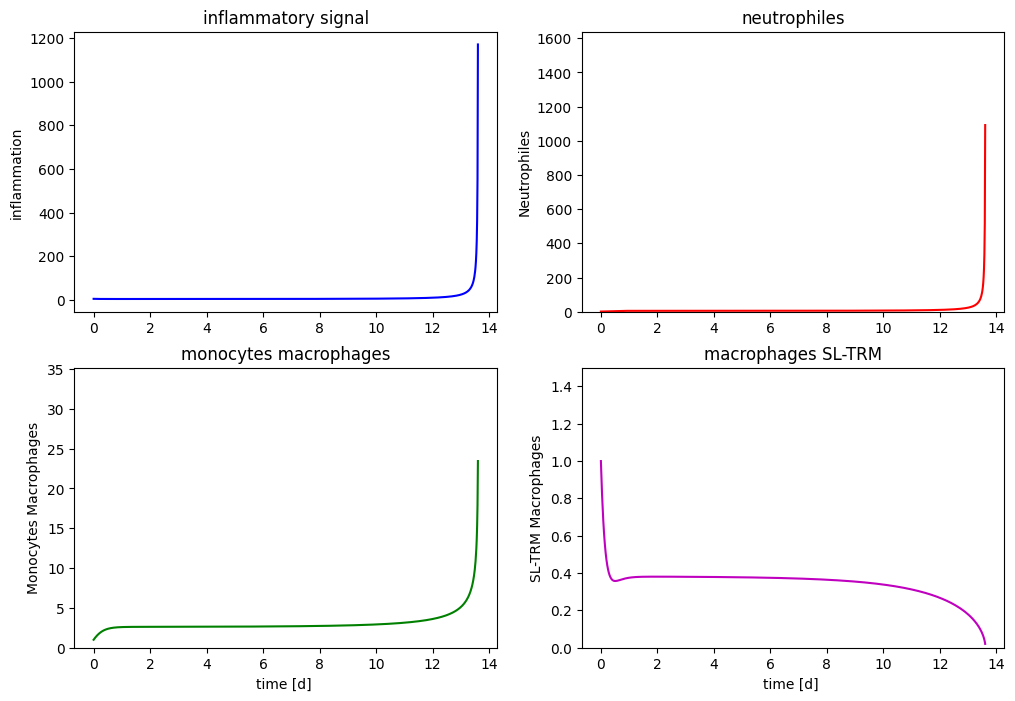}
    \caption{$A=0.073, S(0)=5$}\label{A=0,073}
\end{figure}
To see how the parameter $A=a/b$ controls the healing prognosis
we plot in Fig.~\ref{heal} a $(A,S)$ plane figure where for each value of $A$ and starting value $(S(0)=S,N=1,M=1,T=1)$
 we run the model for three weeks and see how far from the disease-free state the systems ends up.
 We color in green points where $(S,M,N,T)$ is close to $(0,1,1,1)$ up to a given tolerance of 5\%.
 We color in red points which end up far away from the equilibrium.
 \begin{figure}[!htbp]
    \centering
    \includegraphics[scale=0.5]{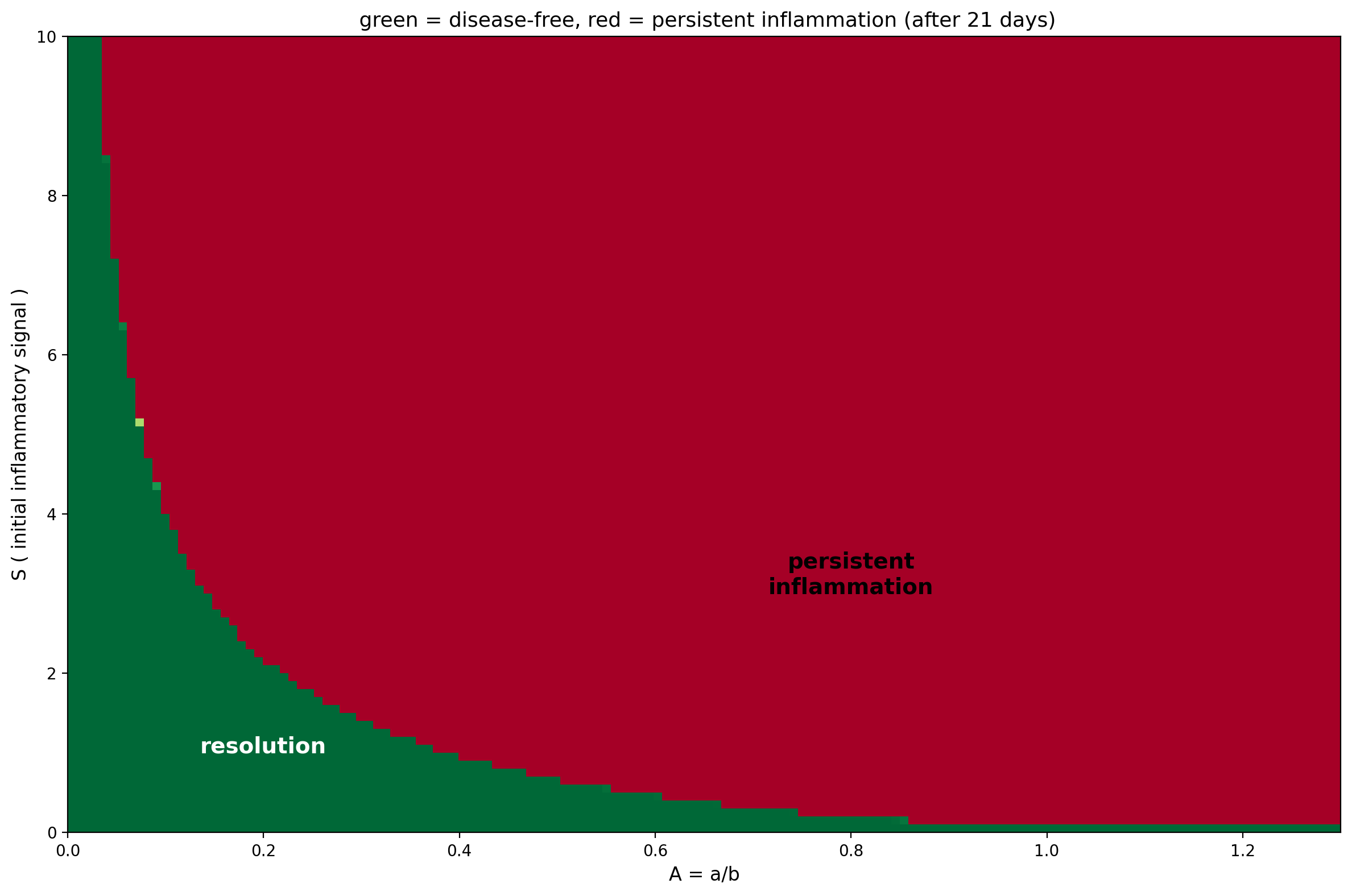}
    \caption{Healing w.r.t. $A=a/b$ and initial dose of immune signal $S$}\label{heal}
\end{figure}
We notice that for $A<1$, the system converges to the disease-free equilibrium provided that the   initial immune signal $S(0)$ is below a given threshold.
 This threshold increases as $A$ decreases, meaning that the system can tolerate higher initial immune signals and still come back  to the disease-free equilibrium when $A$ is small enough.
  This is consistent with the bifurcation diagram in Fig.~\ref{fig:bifurcation} where for $A<1$ the unstable inflammatory equilibrium moves away from the disease-free equilibrium when $A$ decreases.
 For $A\geq 1$, the system diverges away from the disease-free equilibrium for all non zero values of the initial immune signal $S(0)$.
  This suggests that for $A\geq 1$ the system is in a persistent inflammatory state, where the  inflammation is self-sustained and does not resolve, even if the initial immune signal is low.


The Python Jupyter notebook used to perform the simulations is available at \url{https://github.com/azerad/immunomath}.

\section{Biological insights and testable predictions.}\label{sec:insights_predictions}
Although deliberately minimal, 
the model suggests that the outcome of serum-transfer arthritis may depend on a balance between inflammatory amplification and resident macrophage-mediated regulation. 
In this framework, neutrophil-driven amplification promotes the persistence of the inflammatory signal, whereas SL-TRM exert a counter-regulatory effect. 
The model therefore predicts the existence of a threshold separating self-limited inflammation from sustained arthritis. 
This threshold may be shifted by the intensity of the initial arthritogenic stimulus, the efficiency of monocyte/macrophage recruitment, and the integrity of the resident macrophage compartment. 
These predictions could be tested experimentally by time-resolved quantification of neutrophils, monocyte-derived macrophages and resident synovial macrophages during STA, 
combined with perturbations that either reduce neutrophil recruitment or preserve/enhance resident macrophage regulatory functions.

\section{Concluding remarks and perspectives.}\label{sec:conclusion} In this work, we designed a simple model able to describe essential features of the monocyte-macrophages, neutrophiles and sublining tissue-resident macrophages dynamics
 submitted to an inflammatory stimulus.
We identified a single parameter controlling the positive or negative feedbacks processes at stake in arthritis. 
Essentially the ratio between the inflammatory effect of neutrophiles and the anti-inflammatory effect of SL-TRM macrophages appeared critical. 
Our study suggests therefore that drugs acting differently on neutrophiles and SL-TRM macrophages could be effective in curing arthritis.

\section*{Acknowledgments.}
We acknowledge the support of Immun4Cure University Hospital Institute ``Institute for innovative immunotherapies in autoimmune diseases'' [France 2030/ANR-23-IHUA-0009].

\end{document}